\documentclass[11pt]{llncs}
\usepackage[utf8]{inputenc}
\usepackage {amsmath,amsfonts,amssymb}
\usepackage {theorem}
\usepackage {url}
\usepackage{aeguill} 
\usepackage{graphicx}
\usepackage{tikz}
\usepackage{a4wide}

\newcommand{\ignore}[1]{}

\newcommand{\eqn}[1]{\eqref{#1}}

\newcommand{\C}{\mathcal{C}}
\newcommand{\D}{\Delta}
\newcommand{\F}{\mathbb{F}}
\newcommand{\Fbar}{\overline{\mathbb{F}}}
\newcommand{\Q}{\mathbb{Q}}
\newcommand{\Z}{\mathbb{Z}}
\newcommand{\R}{\mathbb{R}}
\newcommand{\M}{\mathsf{M}}
\newcommand{\sO}{{\widetilde{O}}}
 
\newcommand{\Jac}[1]{\ensuremath{J_{#1}}}
\newcommand{\End}{\ensuremath{\mathrm{End}}}
\newcommand{\Pic}[2][{}]{\ensuremath{\mathrm{Pic}^{{#1}}({#2})}}
\newcommand{\variety}[1]{\ensuremath{V\!\left({#1}\right)}}

\newcommand{\dualof}[1]{\ensuremath{{#1}^{\dagger}}}
\newcommand{\family}[1]{\ensuremath{\mathcal{#1}}}
\newcommand{\product}[3][{}]{\ensuremath{{#2}\!\times_{#1}\!{#3}}}
\newcommand{\XxX}[1]{\product{{#1}}{{#1}}}
\newcommand{\isom}{\cong}
\newcommand{\disc}{\mathrm{disc}}

\newcommand{\Mumford}[1]{{({#1})}}
\newcommand{\Mumfordbig}[1]{{\left({#1}\right)}}

\newcommand{\Tr}{\mathrm{Tr}}
\newcommand{\Nr}{\mathrm{N}}

\newcommand{\cc}{\mathfrak{c}}
\newcommand{\pp}{\mathfrak{p}}

\title{Counting Points on Genus 2 Curves\\with Real Multiplication}
\author{P.~Gaudry \and D.~Kohel \and B.~Smith}
\institute{}

\begin{document}

\maketitle

\begin{abstract}
    We present an accelerated Schoof-type point-counting algorithm
    for curves of genus~\(2\) equipped with an efficiently computable
    real multiplication endomorphism.
    Our new algorithm reduces the complexity of genus~\(2\) point
    counting over a finite field \(\F_{q}\) of large characteristic 
    from~\(\sO(\log^8 q)\) to~\(\sO(\log^5 q)\).
    Using our algorithm
    we compute a $256$-bit prime-order Jacobian,
    suitable for cryptographic applications,
    and also the order of a $1024$-bit Jacobian.
\end{abstract}

%%%%%%%%%%%%%%%%%%%%%%%%%%%%%%%%%%%%%%%%%%%%%%%%%%%%%%%%%%%%%%%%%%%%%%%%%%%%%%%%
%%%%%%%%%%%%%%%%%%%%%%%%%%%%%%%%%%%%%%%%%%%%%%%%%%%%%%%%%%%%%%%%%%%%%%%%%%%%%%%%
\section{Introduction}
%%%%%%%%%%%%%%%%%%%%%%%%%%%%%%%%%%%%%%%%%%%%%%%%%%%%%%%%%%%%%%%%%%%%%%%%%%%%%%%%
%%%%%%%%%%%%%%%%%%%%%%%%%%%%%%%%%%%%%%%%%%%%%%%%%%%%%%%%%%%%%%%%%%%%%%%%%%%%%%%%

Cryptosystems based on curves of genus $2$ offer per-bit security
and efficiency comparable with elliptic curve cryptosystems.
However, many of the computational problems
related to creating secure instances of genus~$2$ cryptosystems
are considerably more difficult
than their elliptic curve analogues.
Point counting---or, from a cryptographic point of view, 
computing the cardinality of a cryptographic group---offers 
a good example of this disparity,
at least for curves defined over large prime fields.
Indeed, while computing the order of a cryptographic-sized
elliptic curve with the Schoof--Elkies--Atkin algorithm 
is now routine, 
computing the order of a comparable genus $2$ Jacobian requires a significant 
computational effort~\cite{GaSc04,GaSc10}.

In this article
we describe a number of improvements to the classical Schoof--Pila algorithm
for genus $2$ curves with explicit and efficient real multiplication (RM).
For explicit RM curves over $\F_p$,
we reduce the complexity of Schoof--Pila from $\sO(\log^8 p)$
to $\sO(\log^5 p)$.
We applied a first implementation of our algorithms
to find prime-order Jacobians over 128-bit fields
(comparable to prime-order elliptic curves over 256-bit fields,
and therefore suitable for contemporary cryptographic applications).
Going further, 
we were able to compute the order of an RM Jacobian defined over a 512-bit prime field,
far beyond the cryptographic range.
(For comparison, the previous record computation in genus $2$
was over a 128-bit field.)

While these RM curves are special, they are not ``too special'':
Every ordinary genus $2$ Jacobian over a finite field
has RM; our special requirement is that this RM be known in advance 
and be efficiently computable.
The moduli of curves with RM by
a fixed ring form $2$-dimensional subvarieties (Humbert surfaces)
in the $3$-dimensional moduli space of all genus~$2$ curves.
We can generate random curves with the specified RM by choosing 
random points on an explicit model of the corresponding Humbert
surface~\cite{Gruenewald}. 
In comparison with elliptic curves, 
for which the moduli space is one-dimensional, this 
still gives an additional degree of freedom in the random curve selection.
To generate random curves with efficiently computable RM, we choose 
random curves from some known one and two-parameter families
(see~\S\ref{sec:families}).

Curves with efficiently computable RM 
have an additional benefit in cryptography:
the efficient endomorphism can be used to accelerate scalar multiplication
on the Jacobian, yielding faster encryption and
decryption~\cite{Kohel--Smith,Parketal,Takashima}.
The RM formul\ae{} are also compatible with
fast arithmetic based on theta functions~\cite{Gaudry-theta}.

%%%%%%%%%%%%%%%%%%%%%%%%%%%%%%%%%%%%%%%%%%%%%%%%%%%%%%%%%%%%%%%%%%%%%%%%%%%%%%%%
%%%%%%%%%%%%%%%%%%%%%%%%%%%%%%%%%%%%%%%%%%%%%%%%%%%%%%%%%%%%%%%%%%%%%%%%%%%%%%%%
\section{Conventional Point Counting for Genus $2$ Curves}
%%%%%%%%%%%%%%%%%%%%%%%%%%%%%%%%%%%%%%%%%%%%%%%%%%%%%%%%%%%%%%%%%%%%%%%%%%%%%%%
%%%%%%%%%%%%%%%%%%%%%%%%%%%%%%%%%%%%%%%%%%%%%%%%%%%%%%%%%%%%%%%%%%%%%%%%%%%%%%%

Let $\C$ be a curve of genus 2 over a finite field $\F_q$, of odd 
characteristic, defined by an affine model $y^2 = f(x)$, 
where $f$ is a squarefree polynomial of degree $5$ or $6$ over $\F_q$.
Let $\Jac{\C}$ be the Jacobian of $\C$; we assume $\Jac{\C}$ is ordinary and 
absolutely simple.  Points on $\Jac{\C}$ correspond to
degree-$0$ divisor classes on $\C$;
we use the Mumford representation for divisor classes
together with the usual Cantor-style composition and reduction algorithms
for divisor class arithmetic~\cite{Galbraith--Harrison--Mireles-Morales,Cantor-reduction}.
Multiplication by $\ell$ on $\Jac{\C}$ is denoted by $[\ell]$, and 
its kernel by $\Jac{\C}[\ell]$. More generally, if $\phi$ is an 
endomorphism of $\Jac{\C}$ then $\Jac{\C}[\phi] = \ker(\phi)$,
and if $S$ is a set of endomorphisms then $\Jac{\C}[S]$ denotes 
the intersection of $\ker(\phi)$ for $\phi$ in $S$.

%%%%%%%%%%%%%%%%%%%%%%%%%%%%%%%%%%%%%%%%%%%%%%%%%%%%%%%%%%%%%%%%%%%%%%%%%%%%%%%%
%%%%%%%%%%%%%%%%%%%%%%%%%%%%%%%%%%%%%%%%%%%%%%%%%%%%%%%%%%%%%%%%%%%%%%%%%%%%%%%%
\subsection{The Characteristic Polynomial of Frobenius}
%%%%%%%%%%%%%%%%%%%%%%%%%%%%%%%%%%%%%%%%%%%%%%%%%%%%%%%%%%%%%%%%%%%%%%%%%%%%%%%%
%%%%%%%%%%%%%%%%%%%%%%%%%%%%%%%%%%%%%%%%%%%%%%%%%%%%%%%%%%%%%%%%%%%%%%%%%%%%%%%%
\label{sec:Frobenius-charpoly}

We let $\pi$ denote the Frobenius endomorphism of $\Jac{\C}$,
with Rosati dual $\dualof{\pi}$ (so $\pi\dualof{\pi} = [q]$). 
The characteristic polynomial of $\pi$
has the form
\begin{equation}
\label{eq:Frobenius-charpoly}
  \chi(T) = T^4 - s_1 T^3 + (s_2 + 2q)\, T^2 - q s_1 T + q^2,
\end{equation}
where $s_1$ and $s_2$ are integers, and $s_2$ is a translation 
of the standard definition. 
The polynomial $\chi(T)$ determines the cardinality of $\Jac{\C}(\F_{q^k})$ 
for all $k$: in particular, $\#\Jac{\C}(\F_{q}) = \chi(1)$.
We refer to the determination of $\chi(T)$ as the {\it point counting problem}. 

The polynomial $\chi(T)$ is a {\it Weil polynomial}: 
all of its complex roots lie on the circle $|z| = \sqrt{q}$.
This implies the Weil bounds 
\begin{equation}
\label{eq:s_1s_2-Weil-bounds}
    |s_1| \le 4\sqrt{q}\quad \text{ and }\quad |s_2| \le 4q. 
\end{equation}
However, the possible values of $(s_1,s_2)$
do not fill the whole rectangle specified by the Weil bounds.
R\"uck~\cite[Theorem 1.1]{Ruck90}
shows that in fact $s_1$ and $s_2$ satisfy
\[
	s_1^2 - 4 s_2 \ge 0 \text{ and } s_2 + 4q \ge 2|s_1| ,
\]
so the possible values of $(s_1,s_2)$ are in the following domain:
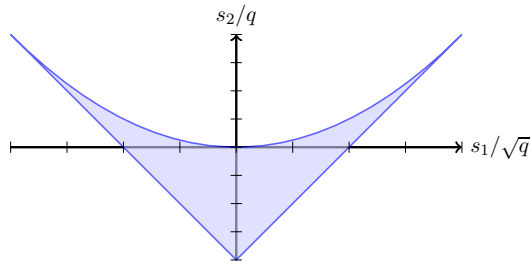
\begin{center}
\scalebox{0.75}{
\begin{tikzpicture}[yscale=0.5]
% \draw[very thin, color=gray] (-4,-4) grid (4,4);
\draw[very thick,->] (-4,0) -- (4,0) node[right] {$s_1/\sqrt{q}$};
\draw[very thick,->] (0,-4) -- (0,4) node[above] {$s_2/q$};
\path[thick,draw=blue, fill=blue!20, opacity=0.6]
(0,-4) -- (-4,4) --
plot[domain=-4:4] (\x,\x*\x/4) -- cycle;
\foreach \x in {-4,...,4} \draw[yscale=2] (\x,-0.1) -- (\x,0.1);
\foreach \x in {-4,...,4} \draw (-0.1,\x) -- (0.1,\x);
\end{tikzpicture}
}
\end{center}

%%%%%%%%%%%%%%%%%%%%%%%%%%%%%%%%%%%%%%%%%%%%%%%%%%%%%%%%%%%%%%%%%%%%%%%%%%%%%%%%
%%%%%%%%%%%%%%%%%%%%%%%%%%%%%%%%%%%%%%%%%%%%%%%%%%%%%%%%%%%%%%%%%%%%%%%%%%%%%%%%
\subsection{The Classical Schoof--Pila Algorithm for Genus $2$ Curves}
%%%%%%%%%%%%%%%%%%%%%%%%%%%%%%%%%%%%%%%%%%%%%%%%%%%%%%%%%%%%%%%%%%%%%%%%%%%%%%%%
%%%%%%%%%%%%%%%%%%%%%%%%%%%%%%%%%%%%%%%%%%%%%%%%%%%%%%%%%%%%%%%%%%%%%%%%%%%%%%%%
\label{sec:Schoof--Pila}

The objective of point counting is to compute $\chi(T)$, or equivalently 
the tuple of integers $(s_1,s_2)$.  When the characteristic of $\F_{q}$ 
is large, the conventional approach %to computing $\chi(T)$
is to apply the Schoof--Pila algorithm 
as far as is practical, before passing to a baby-step giant-step 
algorithm if necessary (see~\S\ref{sec:BSGS}).  
The strategy of Schoof's algorithm 
and its generalizations is to compute the polynomials
$\chi_\ell(T) = \chi(T) \bmod (\ell)$ for sufficiently many primes (or prime 
powers) $\ell$ to reconstruct $\chi(T)$ using the Chinese Remainder Theorem.  

Since $\chi_\ell(T)$ is the characteristic polynomial of $\pi$ 
restricted to $\Jac{\C}[\ell]$ (see~\cite[Proposition 2.1]{Pila}),
we have 
$$
\chi_\ell(\pi)(D) = 0 \text{ for all } D \text{ in } \Jac{\C}[\ell]. 
$$

Conversely, to compute $\chi_\ell(T)$ we let $D$ be a generic 
element of $\Jac{\C}[\ell]$ (as in~\S\ref{sec:kernel-ideal} below),
compute the three points
\[
    (\pi^2 + [\bar{q}])^2(D),\ \ 
    (\pi^2 + [\bar{q}])\pi(D),\ \text{and}\ \   
    \pi^2(D),
\]
and then search for the coefficients $(\bar{s}_1,\bar{s}_2)$ of 
$\chi_\ell(T)$ in $(\Z/\ell\Z)^2$, for which the linear relation
\begin{equation}
\label{eq:conventional-ell-test}
     (\pi^2 + [\bar{q}])^2(D) 
    - [\bar{s}_1]\,(\pi^2 + [\bar{q}])\pi(D) 
    + [\bar{s}_2]\,\pi^2(D) = 0
\end{equation}
holds.
If the minimal polynomial of $\pi$ on $\Jac{\C}[\ell]$ is 
a proper divisor of $\chi_\ell(T)$---which occurs for at most 
a finite number of $\ell$ dividing $\disc(\chi)$---then 
the polynomial so determined is not unique, but 
$\chi_\ell(T)$ can be determined by deducing the correct 
multiplicities of its factors. 

Once we have computed $\chi_\ell(T)$ for sufficiently many $\ell$,
we reconstruct~$\chi(T)$ using the Chinese Remainder Theorem.
The Weil and R\"uck bounds together with
a weak version of the prime number theorem tell us how many 
$\ell$ are required: Pila notes in~\cite[\S1]{Pila}
that the set of $O(\log q)$ primes $\ell < 21\log q$ will suffice.
We analyse the complexity of the classical Schoof--Pila algorithm
in~\S\ref{sec:Schoof--Pila-complexity}.

%%%%%%%%%%%%%%%%%%%%%%%%%%%%%%%%%%%%%%%%%%%%%%%%%%%%%%%%%%%%%%%%%%%%%%%%%%%%%%%%
%%%%%%%%%%%%%%%%%%%%%%%%%%%%%%%%%%%%%%%%%%%%%%%%%%%%%%%%%%%%%%%%%%%%%%%%%%%%%%%%
\subsection{Endomorphisms and Generic Kernel Elements}
%%%%%%%%%%%%%%%%%%%%%%%%%%%%%%%%%%%%%%%%%%%%%%%%%%%%%%%%%%%%%%%%%%%%%%%%%%%%%%%%
%%%%%%%%%%%%%%%%%%%%%%%%%%%%%%%%%%%%%%%%%%%%%%%%%%%%%%%%%%%%%%%%%%%%%%%%%%%%%%%%
\label{sec:kernel-ideal}

We now recall how to contruct an effective version of a generic
$\ell$-torsion element. We present it in a slightly more general setting,
so that we can use this ingredient in the subsequent RM-specific
algorithm. Therefore, we show how to compute with a generic element
of the kernel of some endomorphism $\phi$ of $\Jac{\C}$, whereas $\phi$ is
just $[\ell]$ in the classical algorithm.

\begin{definition}\label{def:explicit}
    Fix an embedding $P \mapsto D_P$ of $\C$ in $\Jac{\C}$.
    We say that an endomorphism $\phi$ of~$\Jac{\C}$
    is \emph{explicit} if we can effectively compute polynomials 
    $d_0,d_1,d_2,e_0,e_1,$ and $e_2$ such that 
    if $P = (x_P,y_P)$ is a generic point of $\C$, then
    the Mumford representation of $\phi(D_P)$ is given by
    \begin{equation}
    \label{eq:endomorphism-image}
        \phi(D_P) 
        = 
        \Mumfordbig{
            x^2 + \frac{d_1(x_P)}{d_2(x_P)}x + \frac{d_0(x_P)}{d_2(x_P)} ,
            y - y_P\left(
                \frac{e_1(x_P)}{e_2(x_P)}x + \frac{e_0(x_P)}{e_2(x_P)}
            \right)
        }
        .
    \end{equation}
    The $d_0$, $d_1$, $d_2$, $e_0$, $e_1$, and $e_2$ 
    are called the \emph{$\phi$-division polynomials}.
\end{definition}

If $\phi$ is an explicit endomorphism, then 
we can use \eqn{eq:endomorphism-image} (extending $\Z$-linearly)
to evaluate $\phi(D)$ for general divisor classes $D$ in $\Jac{\C}$.
In the case $\phi = [\ell]$, the $[\ell]$-division polynomials
are the $\ell$-division polynomials of Cantor~\cite{Cantor-division}.
The $\phi$-division polynomials depend on the choice of embedding $P
\mapsto D_P$;
we will make this choice explicit when computing the $\phi$-division
polynomials for each of our families 
in~\S\ref{sec:families}.

To compute a generic element of $\Jac{\C}[\phi]$,
we generalize the approach of~\cite{GaSc04} 
(which computes generic elements of $\Jac{\C}[\ell]$).
The resulting algorithm is essentially the same as in~\cite[\S3]{GaSc04} 
(except for the parasite computation step, which we omit)
with $\phi$-division polynomials replacing $\ell$-division polynomials,
so we will only briefly sketch it here.

Let $D = \Mumford{x^2 + a_1x + a_0, y - (b_1x + b_0)}$
be (the Mumford representation of) a generic point of $\Jac{\C}$.
We want to compute a triangular ideal $I_\phi$
in $\F_{q}[a_1,a_0,b_1,b_0]$
vanishing on the nonzero elements of $\Jac{\C}[\phi]$.
The element~$D$ equals $D_{(x_1,y_1)} + D_{(x_2,y_2)}$,
where $(x_1,y_1)$ and $(x_2,y_2)$ are generic points of $\C$.
To find a triangular system
of relations on the $a_i$ and $b_i$
such that $D$ is in $\Jac{\C}[\phi]$
we solve for $x_1$, $y_1$, $x_2$, and $y_2$ in 
\[
    \phi(D_{(x_1,y_1)}) = -\phi(D_{(x_2,y_2)}) ,
\]
applying \eqn{eq:endomorphism-image} and using resultants 
computed with the evaluation--interpolation technique 
of~\cite[\S 3.1]{GaSc04}. 
We then resymmetrize as in~\cite[\S 3.2]{GaSc04}
to express the result in terms of the~$a_i$ and~$b_i$.
We can now compute with a ``generic'' element 
$\Mumford{x^2 + a_1x + a_0, y - (b_1x + b_0)}$
of~$\Jac{\C}[\phi]$
by reducing the coefficients modulo~$I_\phi$
after each operation.

Following the complexity analysis of~\cite[\S3.5]{GaSc04},
we can compute a triangular representation for $I_\phi$
in $O(\delta^2\M(\delta)\log\delta + \M(\delta^2)\log\delta)$
field operations,
where $\delta$ is the maximum among the degrees of 
the $\phi$-division polynomials,
and $\M(d)$ is the number of operations
required to multiply polynomials of degree $d$ over $\F_{q}$.
Using asymptotically fast multiplication algorithms,
we can therefore compute~$I_\phi$ in~$\sO(\delta^3)$ field operations.
The degree of $I_\phi$ is in $O(\delta^2)$;
with this triangular representation,
each multiplication modulo $I_\phi$ costs 
$\sO(\delta^2)$ field operations.

%%%%%%%%%%%%%%%%%%%%%%%%%%%%%%%%%%%%%%%%%%%%%%%%%%%%%%%%%%%%%%%%%%%%%%%%%%%%%%%%
%%%%%%%%%%%%%%%%%%%%%%%%%%%%%%%%%%%%%%%%%%%%%%%%%%%%%%%%%%%%%%%%%%%%%%%%%%%%%%%%
\subsection{Complexity of Classical Schoof--Pila Point Counting}
%%%%%%%%%%%%%%%%%%%%%%%%%%%%%%%%%%%%%%%%%%%%%%%%%%%%%%%%%%%%%%%%%%%%%%%%%%%%%%%%
%%%%%%%%%%%%%%%%%%%%%%%%%%%%%%%%%%%%%%%%%%%%%%%%%%%%%%%%%%%%%%%%%%%%%%%%%%%%%%%%
\label{sec:Schoof--Pila-complexity}

\begin{proposition}
    The complexity of the classical Schoof--Pila algorithm 
    for a curve of genus $2$ over $\F_{q}$ is in $\sO((\log q)^8)$.
\end{proposition}
\begin{proof}
    To determine $\chi(T)$,
    we need to compute $\chi_\ell(T)$ for $O(\log q)$ primes $\ell$
    in $O(\log q)$.
    To compute $\chi_\ell(T)$,
    we must first compute the $\ell$-division polynomials, 
    which have degrees in $O(\ell^2)$.
    We then compute the kernel ideal $I_\ell$;
    according to the previous subsection,
    the total cost is in $\sO(\ell^6)$ field operations.
    The cost of checking \eqn{eq:conventional-ell-test}
    against a generic element of $\Jac{\C}[\ell]$ 
    decomposes into the cost of computing Frobenius images of the generic
    element in $\sO(\ell^4\log q)$ and of finding the matching pair $(\bar s_1,
    \bar s_2)$ in $\sO(\ell^5)$ field operations.
    So the total complexity for computing $\chi_\ell(T)$ is
    in $\sO(\ell^4(\ell^2+\log q))$ field operations.
    In terms of bit operations, for each $\ell$ bounded by $O(\log q)$,
    we compute $\chi_\ell(T)$ in time $\sO((\log q)^7)$, and the result
    follows from the addition of these costs for all the different
    $\ell$'s.
\qed
\end{proof}

%%%%%%%%%%%%%%%%%%%%%%%%%%%%%%%%%%%%%%%%%%%%%%%%%%%%%%%%%%%%%%%%%%%%%%%%%%%%%%%%
\subsection{Baby-Step Giant-Step Algorithms}
%%%%%%%%%%%%%%%%%%%%%%%%%%%%%%%%%%%%%%%%%%%%%%%%%%%%%%%%%%%%%%%%%%%%%%%%%%%%%%%%
%%%%%%%%%%%%%%%%%%%%%%%%%%%%%%%%%%%%%%%%%%%%%%%%%%%%%%%%%%%%%%%%%%%%%%%%%%%%%%%%
\label{sec:BSGS}

In practice, computing $\chi_\ell(T)$ with classical Schoof--Pila becomes 
impractical for large values of $\ell$. The usual approach is to carry out 
the Schoof--Pila algorithm to the extent possible, obtaining congruences 
for $s_1$ and $s_2$ modulo some integer $M$, before completing the calculation 
using a generic group algorithm such as baby-step giant-step (BSGS).
Our BSGS algorithm of choice is the low-memory parallelized variant of the 
Matsuo--Chao--Tsuji algorithm~\cite{GaScXX,Matsuo--Chao--Tsuji}.

The bounds in \eqn{eq:s_1s_2-Weil-bounds} imply that the 
search space of candidates for $(s_1,s_2)$ is in $O(q^{3/2})$, and  
a pure BSGS approach finds $(s_1,s_2)$ in time and space $\sO(q^{3/4})$.
However, when we apply BSGS after a partial Schoof--Pila computation,
we obtain a congruence for $(s_1,s_2)$ modulo~$M$. 
If $M < 8q$, then the size of the search space 
is reduced to $O(q^{3/2}/M^2)$, and the complexity for 
finding $(s_1,s_2)$ is reduced to $\sO(q^{3/4}/M)$.
For larger $M$, the value of $s_1$ is fully determined, and the problem 
is reduced to a one-dimensional search space of size $O(q/M)$ for 
which the complexity becomes $\sO(\sqrt{q/M})$.

%%%%%%%%%%%%%%%%%%%%%%%%%%%%%%%%%%%%%%%%%%%%%%%%%%%%%%%%%%%%%%%%%%%%%%%%%%%%%%%%
%%%%%%%%%%%%%%%%%%%%%%%%%%%%%%%%%%%%%%%%%%%%%%%%%%%%%%%%%%%%%%%%%%%%%%%%%%%%%%%%
\section{Point Counting in Genus $2$ with Real Multiplication}
%%%%%%%%%%%%%%%%%%%%%%%%%%%%%%%%%%%%%%%%%%%%%%%%%%%%%%%%%%%%%%%%%%%%%%%%%%%%%%%%
%%%%%%%%%%%%%%%%%%%%%%%%%%%%%%%%%%%%%%%%%%%%%%%%%%%%%%%%%%%%%%%%%%%%%%%%%%%%%%%%
\label{sec:RM}

By assumption, $\Jac{\C}$ is ordinary and simple, so $\chi(T)$ 
is an irreducible polynomial defining a quartic CM-field with 
real quadratic subfield $\Q(\sqrt{\D})$.  We say that $\Jac{\C}$ 
(and~$\C$) has \emph{real multiplication} (RM) by $\Q(\sqrt{\D}\,)$.
For a randomly selected curve, $\D$ is in $O(q)$; but in the sequel 
we consider families of curves with RM by $\Q(\sqrt{\D})$ for small 
$\D$ (= $5$ or $8$), admitting an explicit (in the sense of
Definition~\ref{def:explicit}) endomorphism 
$\phi$ such that 
\begin{equation}
\label{eq:real-subring}
    \Z[\phi] = \Q(\sqrt{\D}) \cap \End(\Jac{\C})
\end{equation}
(that is, $\Z[\phi]$ is the full real subring of $\End(\Jac{\C})$), and
\[
    \disc\left(\Z[\phi]\right) = \D .
\]
We presume that
the trace $\Tr(\phi)$ and norm $\Nr(\phi)$, such that 
$\phi^2 - \Tr(\phi)\phi + \Nr(\phi) = 0$, are known.
We also suppose that $\phi$ is \emph{efficient}, in the following sense:

\begin{definition}
    We say that an explicit endomorphism $\phi$ is
    \emph{efficiently computable}
    if the cost of evaluating $\phi$ at points of $\Jac{\C}(\F_{q})$
    requires only $O(1)$
    field operations
    (comparable to a few group operations in $\Jac{\C}$).
    In practice, 
    this means that 
    the $\phi$-division polynomials have small degree.
\end{definition}

The existence of an efficiently computable $\phi$ and knowledge 
of $\D$ allows us to make significant improvements to each stage 
of the Schoof--Pila algorithm. 
Briefly: 
in~\S\ref{sec:RM-char-poly} we use~$\phi$ to simplify the testing 
procedure for each~$\ell$;
in~\S\ref{sec:split-prime} we show that when~$\ell$ splits in~$\Z[\phi]$, 
we can use $\phi$ to obtain a radical reduction in complexity for 
computing $\chi_\ell(T)$; and 
in~\S\ref{sec:search-space} we show that knowing an effective $\phi$ allows 
us to use many fewer primes $\ell$.

%%%%%%%%%%%%%%%%%%%%%%%%%%%%%%%%%%%%%%%%%%%%%%%%%%%%%%%%%%%%%%%%%%%%%%%%%%%%%%%%
%%%%%%%%%%%%%%%%%%%%%%%%%%%%%%%%%%%%%%%%%%%%%%%%%%%%%%%%%%%%%%%%%%%%%%%%%%%%%%%%
\subsection{The RM Characteristic Polynomial}
%%%%%%%%%%%%%%%%%%%%%%%%%%%%%%%%%%%%%%%%%%%%%%%%%%%%%%%%%%%%%%%%%%%%%%%%%%%%%%%%
%%%%%%%%%%%%%%%%%%%%%%%%%%%%%%%%%%%%%%%%%%%%%%%%%%%%%%%%%%%%%%%%%%%%%%%%%%%%%%%%
\label{sec:RM-Frobenius-charpoly}

Let $\psi = \pi+\pi^\dagger$;
we consider $\Z[\psi]$, 
a subring of the real quadratic subring of $\End(\Jac{\C})$.
The characteristic polynomial of $\psi$ is the {\it real Weil polynomial}
\begin{equation}
\label{eq:real-charpoly}
    \xi(T) = T^2 - s_1 T + s_2;
\end{equation}
the discriminant of $\Z[\psi]$ is $\D_0 = s_1^2 - 4 s_2$. 
The analogue for $(s_1,\D_0)$ of R\"uck's bounds is
\begin{equation}
\label{eq:s_1d_0-Ruck-bounds}
    (|s_1| - 4\sqrt{q})^2 \ge \D_0 = s_1^2 - 4 s_2 \ge 0 .
\end{equation}
Equation \eqn{eq:real-subring} implies that
$\Z[\psi]$ is contained in $\Z[\phi]$,
so there exist integers $m$ and $n$ such that
\begin{equation}
\label{eq:RM-psi-relation}
    \psi = m + n\phi .
\end{equation}
Both $s_1$ and $s_2$ are determined by $m$ and $n$: we have
\begin{equation}
\label{eq:s_1s_2-from-nm}
    s_1 = \Tr(\psi) = 2m + n\Tr(\phi) 
    \quad \text{and} \quad 
    s_2 = \Nr(\psi) = (s_1^2 - n^2\D)/4.
\end{equation}
In fact \(n\) is the conductor of \(\Z[\psi]\) in \(\Z[\phi]\) 
up to sign: \( |n| = [\Z[\phi]:\Z[\psi]] \),
and hence
\[
    \D_0 = \disc(\Z[\psi]) = s_1^2 - 4 s_2 = n^2 \D .
\]
The square root of the bounds 
in \eqn{eq:s_1d_0-Ruck-bounds} gives 
bounds on $s_1$ and~$n$:
\[
    4\sqrt{q} - |s_1| \ge \sqrt{\D_0} = |n| \sqrt{\D} \ge 0 ;
\]
In particular, 
\( |s_1| \le 4\sqrt{q} \)
and
\( |n| \le 4\sqrt{q/\D} \).
Applying the relation in~\eqn{eq:s_1s_2-from-nm},
we have the bounds
\begin{equation}
\label{eq:m-n-bounds}
    |m| \le 2(|\Tr(\phi)| + \sqrt{\D})\sqrt{q/\D}
    \quad \text{and}\quad 
    |n| \le 4\sqrt{q/\D}
    .
\end{equation}
Both $|m|$ and $|n|$ are in $O(\sqrt{q})$.

%%%%%%%%%%%%%%%%%%%%%%%%%%%%%%%%%%%%%%%%%%%%%%%%%%%%%%%%%%%%%%%%%%%%%%%%%%%%%%%%
%%%%%%%%%%%%%%%%%%%%%%%%%%%%%%%%%%%%%%%%%%%%%%%%%%%%%%%%%%%%%%%%%%%%%%%%%%%%%%%%
\subsection{An Efficiently Computable RM Relation}
%%%%%%%%%%%%%%%%%%%%%%%%%%%%%%%%%%%%%%%%%%%%%%%%%%%%%%%%%%%%%%%%%%%%%%%%%%%%%%%%
%%%%%%%%%%%%%%%%%%%%%%%%%%%%%%%%%%%%%%%%%%%%%%%%%%%%%%%%%%%%%%%%%%%%%%%%%%%%%%%%
\label{sec:RM-char-poly}

We can use our efficiently computable endomorphism $\phi$ 
to replace the relation
of \eqn{eq:conventional-ell-test}
with a more efficiently computable alternative.
Multiplying \eqn{eq:RM-psi-relation} through by $\pi$,
we have
\[
    \psi\pi = \pi^2 + [q] = m\pi + n\phi\pi.
\]
We can therefore compute $\bar{m} = m\bmod\ell$ and $\bar{n} = n\bmod\ell$
by letting $D$ be a generic $\ell$-torsion point,
computing the three points
\[
    (\pi^2 + [\bar{q}])(D) ,
    \ \  
    \pi(D) ,
    \ \ \text{and} \ \ 
    \phi\pi(D) ,
\]
and then 
searching for the $\bar{m}$ and $\bar{n}$ in $\Z/\ell\Z$
such that 
\begin{equation}
\label{eq:RM-ell-test}
  (\pi^2 + [\bar{q}])(D) - [\bar{m}]\pi(D) - [\bar{n}]\phi\pi(D) = 0
\end{equation}
holds; we can find such an $\bar{m}$ and $\bar{n}$ in $O(\ell)$ group
operations.

Solving \eqn{eq:RM-ell-test} rather than 
\eqn{eq:conventional-ell-test} has several advantages.
First, computing 
\((\pi^2 + [\bar{q}])(D)\), \(\pi(D)\), and \(\phi\pi(D)\)
requires only two applications of Frobenius,
instead of the four required to compute 
\((\pi^2 + [\bar{q}])^2(D)\),
\((\pi^2 + [\bar{q}])\pi(D)\),
and \(\pi^2(D)\)
(and Frobenius applications are costly in practice).
Moreover, either $s_2$ needs to be determined in $O(q)$, or else
the value of $n$ in \eqn{eq:conventional-ell-test} 
leaves a sign ambiguity for each prime $\ell$, because only 
$n^2 \bmod \ell$ can be deduced from $(\bar{s}_1, \bar{s}_2)$. 
In contrast, \eqn{eq:RM-ell-test} determines $n$ directly.

%%%%%%%%%%%%%%%%%%%%%%%%%%%%%%%%%%%%%%%%%%%%%%%%%%%%%%%%%%%%%%%%%%%%%%%%%%%%%%%%
%%%%%%%%%%%%%%%%%%%%%%%%%%%%%%%%%%%%%%%%%%%%%%%%%%%%%%%%%%%%%%%%%%%%%%%%%%%%%%%%
\subsection{Exploiting Split Primes in $\Q(\sqrt{\D})$}
%%%%%%%%%%%%%%%%%%%%%%%%%%%%%%%%%%%%%%%%%%%%%%%%%%%%%%%%%%%%%%%%%%%%%%%%%%%%%%%%
%%%%%%%%%%%%%%%%%%%%%%%%%%%%%%%%%%%%%%%%%%%%%%%%%%%%%%%%%%%%%%%%%%%%%%%%%%%%%%%%
\label{sec:split-prime}

Let $\Z[\phi] \subset \End(\Jac{\C})$ be an RM order in $\Q(\phi) \isom \Q(\sqrt{\D})$.  
Asymptotically, half of all primes $\ell$ split: $(\ell) = \pp_1 \pp_2$  
in $\Z[\phi]$, where $\pp_1 + \pp_2 = (1)$ (and this carries over to 
prime powers $\ell$). 
This factorization gives a decomposition of the $\ell$-torsion
\[
    \Jac{\C}[\ell] = \Jac{\C}[\pp_1] \oplus \Jac{\C}[\pp_2].
\]
In particular, any $\ell$-torsion point $D$ can be uniquely expressed 
as a sum $D = D_1 + D_2$ where $D_i$ is in $\Jac{\C}[\pp_i]$.

According to the Cohen--Lenstra heuristics~\cite{CohenLenstra}, more 
than 75\% of RM fields have class number 1; 
in each of the explicit RM families in \S\ref{sec:families},
the order $\Z[\phi]$ has class number 1.
All ideals are principal in such an order,
so we may find a generator 
for each of the ideals $\pp_i$.  Furthermore, the following lemma shows 
that we can find a generator which is not too large.

\begin{lemma}
\label{lemma:generator-reduction}
    If $\pp$ is a principal ideal of norm $\ell$ in a real quadratic order $\Z[\phi]$,
    then there exists an effectively computable generator of $\pp$ 
    with coefficients in $O(\sqrt{\ell})$.
\end{lemma}
\begin{proof}
    Let $\alpha$ be a generator of $\pp$, 
    and $\varepsilon$ a fundamental unit of $\Z[\phi]$.  
    Let $\gamma \mapsto \gamma_1$
    and $\gamma \mapsto \gamma_2$
    be the two embeddings of $\Z[\phi]$ in $\R$,
    indexed so that
    $|\alpha_1| \ge |\alpha_2|$ and $|\varepsilon_1| > 1$ 
    (replacing $\varepsilon$ with $\varepsilon^{-1}$ if necessary).
    Then $R = \log(|\varepsilon_1|)$ is the regulator of $\Z[\phi]$. 
    Set $\beta = \varepsilon^{-k}\alpha$,
    where $k = [\log(|\alpha_1/\sqrt{\ell}|)/R]$;
    then $\beta = a + b\phi$ is a new generator for $\pp$ such that 
    \[
        -\frac{1}{2} 
        \le 
        \frac{\log(|\beta_i/\sqrt{\ell}|)}{R} 
        \le 
        \frac{1}{2}\cdot 
     \]   
    From the preceding bounds, 
    \( |\beta_1+\beta_2| = |2a + b\Tr(\phi)| \)
    and
    \( |\beta_1-\beta_2| = |b\sqrt{\D}| \)
    are bounded by $2e^{R/2}\sqrt{\ell}$.  
    Since $\Tr(\phi)$, $\D$ and $R$ are 
    fixed constants, $|a|$ and $|b|$ are in $O(\sqrt{\ell})$.
    The ``effective'' part of the result follows from classical
    algorithms for quadratic fields.
    \qed
\end{proof}

\begin{lemma}
\label{lemma:alpha-division-polynomials}
    Let $\Jac{\C}$ be the Jacobian of a genus $2$ curve
    over a finite field $\F_q$ 
    with an efficiently computable RM endomorphism $\phi$.  
    There exists an algorithm which, given a principal ideal $\pp$ of 
    norm $\ell$ in $\Z[\phi]$, 
    computes an explicit generator $\alpha$ of $\pp$ 
    and the $\alpha$-division polynomials in $O(\ell)$ field operations.
\end{lemma}
\begin{proof}
    By Lemma~\ref{lemma:generator-reduction}, we can compute a generator 
    $\alpha = [a] + [b]\phi$ with $a$ and $b$ in $O(\sqrt{\ell})$.  
    The $[a]$- and $[b]$-division polynomials have degrees in $O(\ell)$, 
    and can be determined in $O(\ell)$ field operations. The division 
    polynomials for the sum $\alpha = [a] + [b]\phi$ require one sum 
    and one application of $\phi$; and since $\phi$ is efficiently
    computable, this increases the division polynomial degrees and 
    computing time by at most a constant factor.
\qed
\end{proof}
We can now state the main theorem for RM point counting. 
\begin{theorem}
\label{theorem:RM-ptc-complexity}
    There exists an algorithm for the point counting problem in a family 
    of genus~2 curves with efficiently computable RM of class number 1,
    whose complexity is in $\sO((\log q)^5)$.
\end{theorem}
\begin{proof}
    Let $\Jac{\C}$ be a Jacobian in a family with efficiently computable 
    RM by $\Z[\phi]$. 
    Suppose that $\ell$ is prime, $(\ell) = \pp_1 \pp_2$ in $\Z[\phi]$, 
    and that the $\pp_i$ are principal.
    By Lemma~\ref{lemma:alpha-division-polynomials} we can compute 
    representative $\alpha$-division polynomials for $\Jac{\C}$ for 
    each $\pp$ in $\{\pp_1,\pp_2\}$ in time $\sO(\ell)$, hence 
    generic points $D_i$ in $\Jac{\C}[\pp_i]$.

    We recall that \eqn{eq:RM-ell-test} is the homomorphic 
    image under $\pi$ of the equation 
    $$
    \psi(D) - [\bar{m}](D) - [\bar{n}]\phi(D) = 0.
    $$
    When applied to $D_i$ in $\Jac{\C}[\pp_i]$, both $\psi$ and $\phi$ 
    act as elements of $\Z[\phi]/\pp_i \isom \Z/\ell\Z$. Moreover 
    $\bar{x}_i = \phi \bmod \pp_i$ is known, and it remains to determine 
    $\bar{y}_i = \psi \bmod \pp_i$ by means of 
    the discrete logarithm
    $$
    \psi(D_i) = [\bar{y}_i](D_i) = [\bar{m}+\bar{n}\bar{x}_i](D_i)
    $$
    in the cyclic group $\langle D_i \rangle \isom \Z/\ell\Z$.  
    The application of $\pi$ transports this discrete logarithm 
    problem to that of solving for $\bar{y}_i$ in
    $$
    D_i'' = [\bar{y}_i]D',
    $$
    where $D_i' = \pi(D_i)$ and $D_i'' = (\pi^2+[\bar{q}])(D_i)$.
    By the CRT, from $(\bar{y}_1,\bar{y}_2)$ in $(\Z/\ell\Z)^2$ 
    we recover $\bar{y}$ in $\Z[\phi]/(\ell)$, from which we solve 
    for $(\bar{m},\bar{n})$ in $(\Z/\ell\Z)^2$ such that 
    $$
    \bar{y} = \bar{m} + \bar{n}\phi \in \Z[\phi]/(\ell).
    $$
    The values of $(\bar{s}_1,\bar{s}_2)$ are then recovered from 
    \eqn{eq:s_1s_2-from-nm}.

    The ring $\Z[\phi]$ is fixed, so as $\log q$ goes to infinity
    we find that 50\% of all primes $\ell$ split in $\Z[\phi]$
    by the Chebotarev density theorem.
    It therefore suffices to consider split 
    primes in $O(\log q)$. 
    In comparison with the conventional algorithm
    presented in \S\ref{sec:Schoof--Pila}, we reduce from 
    computation modulo the ideal for $\Jac{\C}[\ell]$ of degree in
    $O(\ell^4)$, to computation 
    modulo the ideals for $\Jac{\C}[\pp_i]$ of degree in $O(\ell^2)$.
    This means a reduction from $\sO(\ell^4(\ell^2+\log q))$ 
    to $\sO(\ell^2(\ell+\log q))$ field operations for the determination of 
    each $\chi_{\ell}(T)$,  
    giving the stated reduction in 
    total complexity from $\sO((\log q)^8)$ to $\sO((\log q)^5)$.
\qed
\end{proof}

\begin{remark}
\label{remark:RM-Schoof-Pila-reduction}
    Computing $(m,n)$ instead of $(s_1,s_2)$ allows us to reduce the number 
    of primes~$\ell$ to be considered by about a half, since
    by~\eqn{eq:m-n-bounds} their product 
    needs to be in $O(\sqrt{q})$ instead of $O(q)$.  
    While this changes only the constant in the asymptotic complexity 
    of the algorithm, it yields a significant improvement in practice.
\end{remark}

\begin{remark}
\label{remark:RM-class-number-not-1}
If $\Z[\phi]$ does not have class number $1$, and if $(\ell) = \pp_1\pp_2$ 
where the $\pp_i$ are not principal, then we may 
use a small complementary ideal $(c) = \cc_1\cc_2$ such that $\cc_i\pp_i$ 
are principal in order to apply Lemma~\ref{lemma:alpha-division-polynomials}
to a larger proportion of small ideals.  
Moreover, if $(\bar{m},\bar{n})$ is known modulo $c$, this can be used 
to reduce the discrete log problem modulo $\ell$. 
Again, since a fixed positive density $1/2h$ of primes are both split 
and principal, where $h$ is the class number of $\Z[\phi]$, this does 
not affect the asymptotic complexity.  
Moreover, the first occurrence of a nontrivial class group is for 
$\D = 65$, beyond the current range for which an explicit RM 
construction is currently known.
\end{remark}

%%%%%%%%%%%%%%%%%%%%%%%%%%%%%%%%%%%%%%%%%%%%%%%%%%%%%%%%%%%%%%%%%%%%%%%%%%%%%%%%
%%%%%%%%%%%%%%%%%%%%%%%%%%%%%%%%%%%%%%%%%%%%%%%%%%%%%%%%%%%%%%%%%%%%%%%%%%%%%%%%
\subsection{Shrinking the BSGS Search Space}
%%%%%%%%%%%%%%%%%%%%%%%%%%%%%%%%%%%%%%%%%%%%%%%%%%%%%%%%%%%%%%%%%%%%%%%%%%%%%%%%
%%%%%%%%%%%%%%%%%%%%%%%%%%%%%%%%%%%%%%%%%%%%%%%%%%%%%%%%%%%%%%%%%%%%%%%%%%%%%%%%
\label{sec:search-space}

In the context of the conventional Schoof-Pila algorithm, we need to 
find $s_1$ in $O(\sqrt{q})$ and $s_2$ in $O(q)$.
However, \eqn{eq:RM-psi-relation}, and the effective form of 
\eqn{eq:RM-ell-test} (valid for all points $D$ of $\Jac{\C}$), 
replaces the determination of $(s_1,s_2)$ with the tuple $(m,n)$ of 
integers in $O(\sqrt{q})$.  
As a result, the search space is reduced from $O(q^{3/2})$ to $O(q)$.
Thus the BSGS strategy can find $(m,n)$ (which determines $(s_1,s_2)$) 
in time and space $O(\sqrt{q})$, compared with $O(q^{3/4})$ when searching 
directly for $(s_1,s_2)$.

As in the general case, if one knows $(m,n)$ modulo an integer $M$, then 
the area of the search rectangle is reduced by a factor of $M^2$, so we 
find the tuple $(m,n)$ in $O(\sqrt{q}/M)$ group operations. 
Contrary to the general case of \S\ref{sec:BSGS}, since $m$ and $n$ have 
the same order of magnitude, the speed-up is always by a factor of $M$.

%%%%%%%%%%%%%%%%%%%%%%%%%%%%%%%%%%%%%%%%%%%%%%%%%%%%%%%%%%%%%%%%%%%%%%%%%%%%%%%%
%%%%%%%%%%%%%%%%%%%%%%%%%%%%%%%%%%%%%%%%%%%%%%%%%%%%%%%%%%%%%%%%%%%%%%%%%%%%%%%%
\section{Examples of Families of Curves with Explicit RM}
%%%%%%%%%%%%%%%%%%%%%%%%%%%%%%%%%%%%%%%%%%%%%%%%%%%%%%%%%%%%%%%%%%%%%%%%%%%%%%%%
%%%%%%%%%%%%%%%%%%%%%%%%%%%%%%%%%%%%%%%%%%%%%%%%%%%%%%%%%%%%%%%%%%%%%%%%%%%%%%%%
\label{sec:families}

We now exhibit some families of curves and efficient RM endomorphisms
that can be used as sources of inputs to our algorithm.

%%%%%%%%%%%%%%%%%%%%%%%%%%%%%%%%%%%%%%%%%%%%%%%%%%%%%%%%%%%%%%%%%%%%%%%%%%%%%%%%
%%%%%%%%%%%%%%%%%%%%%%%%%%%%%%%%%%%%%%%%%%%%%%%%%%%%%%%%%%%%%%%%%%%%%%%%%%%%%%%%
\subsection{Correspondences and Endomorphisms}
%%%%%%%%%%%%%%%%%%%%%%%%%%%%%%%%%%%%%%%%%%%%%%%%%%%%%%%%%%%%%%%%%%%%%%%%%%%%%%%%
%%%%%%%%%%%%%%%%%%%%%%%%%%%%%%%%%%%%%%%%%%%%%%%%%%%%%%%%%%%%%%%%%%%%%%%%%%%%%%%%
\label{sec:correspondences}

To give a concrete representation for endomorphisms of $\Jac{\C}$,
we use \emph{correspondences}: 
that is, divisors on the surface $\XxX{\family{C}}$.
Suppose that $\family{R}$ is a curve on $\XxX{\family{C}}$,
and 
let $\pi_1: \family{R} \to \family{C}$ and $\pi_2: \family{R} \to \family{C}$
be the restrictions to $\family{R}$ 
of the natural projections from $\XxX{\family{C}}$
onto its first and second factors.
We have a pullback homomorphism
\(
    (\pi_1)^*: \Pic{\family{C}} \to \Pic{\family{R}}
\),
defined by
\[
    (\pi_1)^*\Big(\Big[
        \sum_{P \in \family{C}(\Fbar_{q})}\!\!\! n_P P\ 
    \Big]\Big)
    =
    \Big[
        \sum_{P \in \family{C}(\Fbar_{q})}\!\!\! n_P\!\!\!
        \sum_{Q \in \pi_{1}^{-1}(P)}\!\!\! Q\ 
    \Big]
    ,
\]
where the preimages $Q$ are counted with the appropriate multiplicities.
(A standard moving lemma shows that 
we can always choose divisor class representatives
so that each $\pi^{-1}(P)$ is zero-dimensional.)
We also have a pushforward homomorphism
\(
    (\pi_2)_*: \Pic{\family{R}} \to \Pic{\family{C}}
\),
defined by
\[
    (\pi_2)_*\Big(\Big[
        \sum_{Q \in \family{R}(\Fbar_{q})}\!\!\! n_Q Q\ 
    \Big]\Big)
    =  
    \Big[ \sum_{Q \in \family{R}(\Fbar_{q})}\!\!\! n_Q \pi_2(Q)\  \Big]
    .
\]
Note that $(\pi_1)^*$ maps $\Pic[n]{\family{C}}$ 
into $\Pic[(n\deg\pi_1)]{\family{R}}$
and $(\pi_2)_*$ maps $\Pic[n]{\family{R}}$
into $\Pic[n]{\family{C}}$
for all $n$.
Hence $(\pi_2)_*\circ(\pi_1)^*$
maps $\Pic[0]{\family{C}}$ into
$\Pic[0]{\family{C}}$,
so we have an \emph{induced endomorphism}
\[
    \phi = (\pi_2)_*\circ(\pi_1)^* : \Jac{\family{C}} \to \Jac{\family{C}} 
    .
\]

We write $x_1,y_1$ and $x_2,y_2$ for the coordinates on the first and second
factors of $\XxX{\family{C}}$, respectively
(so $\pi_i(x_1,y_1,x_2,y_2) = (x_i,y_i)$).
In our examples,
the correspondence
$\family{R}$ 
will be defined by two equations:
\[
    \family{R} 
    = 
    \variety{A(x_1,x_2),B(x_1,y_1,x_2,y_2)}
    .
\]
On the level of divisors,
the image of a generic point $P = (x_P,y_P)$ of $\family{C}$
(that is, a generic prime divisor)
under the endomorphism
\(\phi\)
is given by
\[
    \phi: 
    (x_P,y_P)
    \longmapsto
    \variety{A(x_P,x),B(x_P,y_P,x,y)}
    .
\]
Using
the relations $y_P^2 = f(x_P)$ and $y^2 = f(x)$
(and the fact that correspondences cut out by principal ideals
induce the zero homomorphism),
we can easily replace $A$ and $B$ with 
Cantor-reducible generators 
to 
derive the Mumford representation of $\phi(P)$,
and thus the $\phi$-division polynomials.

%%%%%%%%%%%%%%%%%%%%%%%%%%%%%%%%%%%%%%%%%%%%%%%%%%%%%%%%%%%%%%%%%%%%%%%%%%%%%%%%
%%%%%%%%%%%%%%%%%%%%%%%%%%%%%%%%%%%%%%%%%%%%%%%%%%%%%%%%%%%%%%%%%%%%%%%%%%%%%%%%
\subsection{A $1$-dimensional Family with RM by $\Z[\big(1+\sqrt{5}\big)/2]$}
%%%%%%%%%%%%%%%%%%%%%%%%%%%%%%%%%%%%%%%%%%%%%%%%%%%%%%%%%%%%%%%%%%%%%%%%%%%%%%%%
%%%%%%%%%%%%%%%%%%%%%%%%%%%%%%%%%%%%%%%%%%%%%%%%%%%%%%%%%%%%%%%%%%%%%%%%%%%%%%%%
\label{sec:CTfamily}

\newcommand{\TTV}[1]{\ensuremath{{#1}_{\mathrm{T}}}}

Let $t$ be a free parameter,
and suppose that $q$ is not a power of $5$.
Let $\TTV{\family{C}}$ be the family of curves of genus $2$ over $\F_{q}$
considered by Tautz, Top, and Verberkmoes
in~\cite[Example 3.5]{Tautz--Top--Verberkmoes},
defined by
\[
    \TTV{\family{C}}: y^2 = x^5 - 5x^3 + 5x + t .
\]
Let $\tau_5 = \zeta_5 + \zeta_5^{-1}$,
where $\zeta_5$ is a $5$th root of unity in $\Fbar_{q}$.
Let
$\TTV{\phi}$
be the endomorphism induced by
the (constant) family of 
correspondences 
\[
    \TTV{\family{R}} 
    = 
    \variety{ x_1^2 + x_2^2 - \tau_5 x_1x_2 + \tau_5^2 - 4 , y_1 - y_2 }
    \subset \XxX{\TTV{\family{C}}}
    .
\]
(Note that \(\TTV{\family{R}}\) and $\TTV{\phi}$
are defined over $\F_{q}(\tau_5)$,
which is equal to $\F_{q}$ 
if and only if $q \not\equiv\pm2\bmod5$.)
The family $\TTV{\family{C}}$ has an unique point $P_\infty$ at
infinity,
which we can use to define an embedding
\[
    P = (x_P,y_P) 
    \longmapsto 
    D_P := [(P) - (P_\infty)]
    \leftrightarrow 
    (x - x_P, y - y_P) 
\]
of $\TTV{\family{C}}$ in $\Jac{\TTV{\family{C}}}$.
With respect to this embedding,
the $\TTV{\phi}$-division polynomials are 
\[
    d_2 = 1 ,
    \quad 
    d_1 = -\tau_5 x ,
    \quad 
    d_0 = x^2 + \tau_5^2 - 4 ,
    \quad 
    e_2 = 1,  
    \quad 
    e_1 = 0,  
    \quad 
    e_0 = 1. 
\]

\begin{proposition}
\label{proposition:TTV}
    The minimal polynomial of \(\TTV{\phi}\) 
    is \(T^2 + T - 1\):
    that is, \(\TTV{\phi}\)
    acts as multiplication by $-(1 + \sqrt{5})/2$
    on~\(\Jac{\TTV{\family{C}}}\).
    A prime \(\ell\)
    splits into two principal ideals in \(\Z[\TTV{\phi}]\)
    if and only if \(\ell \equiv \pm1\bmod 5\).
\end{proposition}
\begin{proof}
    The first claim is proven
    in~\cite[\S3.5]{Tautz--Top--Verberkmoes}.
    More directly,
    if \(P\) and \(Q\)
    are generic points of \(\TTV{\family{C}}\),
    then on the level of divisors we find
    \[
        (\TTV{\phi}^2 + \TTV{\phi})((P) - (Q)) 
        = 
        (P) - (Q) + \mathrm{div}\left(\frac{y - y(P)}{y - y(Q)}\right).
    \]    
    Hence \(\Z[\TTV{\phi}]\)
    is isomorphic to the ring of integers of \(\Q(\sqrt{5})\).
    The primes~\(\ell\) splitting in \(\Q(\sqrt{5})\)
    are precisely those congruent to \(\pm 1\) modulo \(5\);
    and \(\Q(\sqrt{5})\) has class number \(1\),
    so the primes over \(\ell\) are principal.
    \qed
\end{proof}

The Igusa invariants of $\TTV{\family{C}}$,
viewed as a point in weighted projective space, are
\( ( 140 : 550 : 20(32t^2 - 3) : 25(896t^2 - 3109) : 64(t^2-4)^2) \);
in particular,
$\TTV{\family{C}}$
has a one-dimensional image in the moduli space of curves of genus~$2$.
The Jacobian of the curve with the same defining equation over $\Q(t)$
is absolutely simple (cf.~\cite[Remark 15]{Kohel--Smith}).

%%%%%%%%%%%%%%%%%%%%%%%%%%%%%%%%%%%%%%%%%%%%%%%%%%%%%%%%%%%%%%%%%%%%%%%%%%%%%%%%
%%%%%%%%%%%%%%%%%%%%%%%%%%%%%%%%%%%%%%%%%%%%%%%%%%%%%%%%%%%%%%%%%%%%%%%%%%%%%%%%
\subsection{A $2$-dimensional Family with RM by $\Z[\big(1+\sqrt{5}\big)/2]$}
%%%%%%%%%%%%%%%%%%%%%%%%%%%%%%%%%%%%%%%%%%%%%%%%%%%%%%%%%%%%%%%%%%%%%%%%%%%%%%%%
%%%%%%%%%%%%%%%%%%%%%%%%%%%%%%%%%%%%%%%%%%%%%%%%%%%%%%%%%%%%%%%%%%%%%%%%%%%%%%%%
\label{sec:MestreFive}

\newcommand{\MestreFive}[1]{\ensuremath{{#1}_{\mathrm{H}}}}

Let $s$ and $t$ be free parameters,
and consider the family of curves
\( \MestreFive{\family{C}} : y^2 = \MestreFive{F}(x) \),
where
\[
    \MestreFive{F}(x)
    =
    sx^5 - (2s + t)x^4 + (s^2 + 3s + 2t - 1)x^3 - (3s + t - 3)x^2 + (s - 3)x + 1
    .
\]
This family is essentially due to Humbert;
it is equal to the family of Mestre~\cite[\S2.1]{Mestre}
with \( (U,T) = (s,t) \),
and the family of Wilson~\cite[Proposition 3.4.1]{Wilson}
with \( (A,B) = (s,-t-3s+3) \).
The family has a full $2$-dimensional image
in the moduli space of genus $2$ curves.

Let $\MestreFive{\family{R}}$ be the family of correspondences
on \( \XxX{\MestreFive{\family{C}}} \)
defined by
\[
    \MestreFive{\family{R}}
    =
    \variety{ x_1^2x_2^2 + s(s-1)x_1x_2 - s^2(x_1 - x_2) + s^2 , y_1 - y_2 }
    ;
\]
let \(\MestreFive{\phi}\)
be the induced endomorphism.
The family $\MestreFive{\family{C}}$ has a
unique point $P_\infty$ at infinity,
which we can use to 
define an embedding
\[
    P = (x_P,y_P) 
    \longmapsto 
    D_P := [(P) - (P_\infty)]
    \leftrightarrow 
    (x - x_P, y - y_P) 
\]
of $\MestreFive{\family{C}}$ in $\Jac{\MestreFive{\family{C}}}$.
With respect to this embedding,
the \(\MestreFive{\phi}\)-division polynomials
are 
\[
    d_2 = x^2,
    \ \ 
    d_1 = (s^2-s)x + s^2,
    \ \ 
    d_0 = -s^2x + s^2, 
    \ \ 
    e_2 = 1,
    \ \ 
    e_1 = 0,
    \ \ 
    e_0 = 1.
\]
\begin{proposition}
\label{proposition:MestreFive}
    The minimal polynomial of \(\MestreFive{\phi}\)
    is \(T^2 + T - 1\):
    that is, \(\MestreFive{\phi}\)
    acts as multipliction by $-(1 + \sqrt{5})/2$
    on~$\Jac{\MestreFive{\family{C}}}$.
    A prime \(\ell\)
    splits into two principal ideals in \(\Z[\MestreFive{\phi}]\)
    if and only if \(\ell \equiv \pm 1 \bmod 5\).
\end{proposition}
\begin{proof}
    The first assertion is~\cite[Proposition 2]{Mestre}
    with \(n = 5\);
    the rest of the proof is exactly as in
    Proposition~\ref{proposition:TTV}.
    \qed
\end{proof}    

%%%%%%%%%%%%%%%%%%%%%%%%%%%%%%%%%%%%%%%%%%%%%%%%%%%%%%%%%%%%%%%%%%%%%%%%%%%%%%%%
%%%%%%%%%%%%%%%%%%%%%%%%%%%%%%%%%%%%%%%%%%%%%%%%%%%%%%%%%%%%%%%%%%%%%%%%%%%%%%%%
\subsection{A $2$-dimensional Family with RM by $\Z[\sqrt{2}]$}
%%%%%%%%%%%%%%%%%%%%%%%%%%%%%%%%%%%%%%%%%%%%%%%%%%%%%%%%%%%%%%%%%%%%%%%%%%%%%%%%
%%%%%%%%%%%%%%%%%%%%%%%%%%%%%%%%%%%%%%%%%%%%%%%%%%%%%%%%%%%%%%%%%%%%%%%%%%%%%%%%
\label{sec:MestreTwo}

\newcommand{\MestreTwo}[1]{\ensuremath{{#1}_{\mathrm{M}}}}

For an example with $\Delta = 8$,
we present a twisted and reparametrized version
of a construction due to Mestre~\cite{Mestre-09}.
Let $s$ and $t$ be free parameters,
let $v(s)$ and $n(s)$ be the rational functions
\[
    v = v(s) := \frac{s^2+2}{s^2-2} 
    \quad \text{and}\quad 
    n = n(s) := \frac{4s(s^4+4)}{(s^2-2)^3} ,
\]
and let $\MestreTwo{\family{C}}$ be the family of curves
defined by
\[
    \MestreTwo{\family{C}} 
    : 
    y^2 
    = 
    \MestreTwo{F}(x)
    := 
    (vx - 1)(x-v)(x^4 - tx^2 + vt - 1) 
    .
\]
The family 
of correspondences 
on \(\XxX{\MestreTwo{\family{C}}}\)
defined by
\[
    \MestreTwo{\family{R}}
    =
    V\left(\begin{array}{l}
        x_1^2x_2^2 - v^2(x_1^2+x_2^2) + 1,
        \\
        y_1y_2-n(x_1^2+x_2^2-t)(x_1x_2-v(x_1+x_2)+1)
    \end{array}\right)
\]
induces an endomorphism 
\( \MestreTwo{\phi} \)
of $\Jac{\MestreTwo{\family{C}}}$.

The family $\MestreTwo{\family{C}}$
has two points at infinity, $P_\infty^+$ and $P_\infty^-$.
which are generically only defined over a quadratic extension of
$\F_{q}(s,t)$.
Let $D_\infty = (P_\infty^+) + (P_\infty^-)$ 
denote the divisor at infinity.
We can use the rational Weierstrass point $P_v = (v,0)$ 
on $\MestreTwo{\family{C}}$ 
to define an embedding
\[
    P = (x_P,y_P)
    \longmapsto
    D_P :=
    [ (P) + (P_v) - D_\infty ]
    \leftrightarrow
    \Big((x-x_P)(x - v), y - \frac{y_P}{x_P-v}(x-v)\Big)
\]
of $\MestreTwo{\family{C}}$ in $\Jac{\MestreTwo{\family{C}}}$
(the appropriate composition and reduction algorithms for
divisor class arithmetic on genus 2 curves with an even-degree model,
such as $\Jac{\MestreTwo{\family{C}}}$,
appear in~\cite{Galbraith--Harrison--Mireles-Morales}.)
With respect to this embedding,
the $\MestreTwo{\phi}$-division polynomials are 
\[
\begin{array}{r@{\;=\;}l@{\qquad}r@{\;=\;}l}
    d_2  &  x^2 - v^2  ,
    &
    e_2  &  (x^2 - v^2)\MestreTwo{F}(x) , % FIXME: check this one  -b
    \\
    d_1  &  0 ,
    &
    e_1  &  n(x-v)(x^4 - tx^2 + tv^2 - 1) ,
    \\
    d_0  &  -v^2x^2 + 1 ,
    &
    e_0  &  n(vx - 1)(x^4 - tx^2 + tv^2 - 1)  .
\end{array}
\]

\begin{proposition}
\label{proposition:MestreTwo}
    The minimal polynomial of \(\MestreTwo{\phi}\)
    is \(T^2 - 2\):
    that is, 
    \(\MestreTwo{\phi}\) acts as multiplication by $\sqrt{2}$ 
    on $\Jac{\MestreTwo{\family{C}}}$.
    A prime \(\ell\)
    splits into two principal ideals in \(\Z[\MestreTwo{\phi}]\)
    if and only if \(\ell\equiv\pm1\bmod{8}\).
\end{proposition}
\begin{proof}
    Let \(P\) and \(Q\)
    be generic points of \MestreTwo{\family{C}}.
    An elementary but lengthy calculation shows that
    on the level of divisors,
    \[
        \MestreTwo{\phi}^2((P) - (Q))
        =
        2(P)-2(Q) + \mathrm{div}\left(\frac{x+x(P)}{x+x(Q)}\right)
        ,
    \]
    so
    \(\MestreTwo{\phi}^2([D]) = 2[D]\) for all \([D]\) in
    \(\Pic[0]{\MestreTwo{\family{C}}}\).
    Hence \(\MestreTwo{\phi}^2 = [2]\),
    and $\Z[\MestreTwo{\phi}]$ is isomorphic to
    the maximal order of $\Q(\sqrt{2})$.
    The primes \(\ell\) splitting in $\Q(\sqrt{2})$
    are precisely those congruent to \(\pm 1\) modulo \(8\);
    further, $\Q(\sqrt{2})$ has class number 1,
    so the primes over \(\ell\) are principal.
    \qed
\end{proof}
\begin{remark}
    As noted above,
    this construction is a twisted reparametrization
    of a family of isogenies described by
    Mestre in~\cite[\S2.1]{Mestre-09}.
    Let \(a_1\) and \(a_2\)
    be the roots of $T^2 - tT + v^2t-1$ in $\overline{\F_{q}(v,t)}$.
    Mestre's curves \(C'\) and \(C\)
    are equal 
    (over \(\F_{q}(v,a_1,a_2)\))
    to 
    our \(\MestreTwo{\family{C}}\) 
    and its quadratic twist by 
    \(A = 2(v^2-1)(v^2+1)^2 = (2n)^2 \),
    respectively.
    We may specialize the proofs in~\cite{Mestre-09}
    to show that $\MestreTwo{\family{C}}$
    has a two-dimensional image in the moduli space of curves of genus $2$,
    and that the Jacobian of the curve with the same defining equation over $\Q(s,t)$
    is absolutely simple.
    Constructions of curves with RM by \(\Z[\sqrt{2}]\)
    are further investigated in Bending's thesis~\cite{Bending}.
\end{remark}

\begin{remark}
    The algorithms described here should be readily adaptable
    to work with Kummer surfaces instead of Jacobians.
    In the notation of~\cite{Gaudry-theta},
    the Kummers with parameters $(a,b,c,d)$ 
    satisfying $b^2 = a^2-c^2-d^2$
    have RM by $\Z[\sqrt{2}]$,
    which can be made explicit as follows: 
    the doubling algorithm decomposes into two identical steps, 
    since $(A:B:C:D) = (a:b:c:d)$, 
    and the components after one step are 
    the coordinates of a Kummer point.
    The step therefore defines an efficiently computable endomorphism 
    which squares to give multiplication by 2.
\end{remark}    

%%%%%%%%%%%%%%%%%%%%%%%%%%%%%%%%%%%%%%%%%%%%%%%%%%%%%%%%%%%%%%%%%%%%%%%%%%%%%%%%
%%%%%%%%%%%%%%%%%%%%%%%%%%%%%%%%%%%%%%%%%%%%%%%%%%%%%%%%%%%%%%%%%%%%%%%%%%%%%%%%
\section{Numerical Experiments}
%%%%%%%%%%%%%%%%%%%%%%%%%%%%%%%%%%%%%%%%%%%%%%%%%%%%%%%%%%%%%%%%%%%%%%%%%%%%%%%%
%%%%%%%%%%%%%%%%%%%%%%%%%%%%%%%%%%%%%%%%%%%%%%%%%%%%%%%%%%%%%%%%%%%%%%%%%%%%%%%%

We implemented our algorithm in C++ using the NTL library. 
For non-critical steps, including computations in quadratic fields, we
used Magma for simplicity. With this implementation, 
the determination of $\chi(T)$ for a curve over a 128-bit prime field 
takes approximately 3 hours on one core of a Core2 processor at 2.83 GHz. 
This provides a proof of concept rather than an optimized implementation. 

%%%%%%%%%%%%%%%%%%%%%%%%%%%%%%%%%%%%%%%%%%%%%%%%%%%%%%%%%%%%%%%%%%%%%%%%%%%%%%%%
%%%%%%%%%%%%%%%%%%%%%%%%%%%%%%%%%%%%%%%%%%%%%%%%%%%%%%%%%%%%%%%%%%%%%%%%%%%%%%%%
\subsection{Cryptographic Curve Generation}
%%%%%%%%%%%%%%%%%%%%%%%%%%%%%%%%%%%%%%%%%%%%%%%%%%%%%%%%%%%%%%%%%%%%%%%%%%%%%%%%
%%%%%%%%%%%%%%%%%%%%%%%%%%%%%%%%%%%%%%%%%%%%%%%%%%%%%%%%%%%%%%%%%%%%%%%%%%%%%%%%

When looking for a cryptographic curve, we used an early-abort 
strategy, where we switch to another curve as soon as either the 
order of the Jacobian order or its twist can not be prime. 
Using our adapted version of Schoof algorithm, we guarantee that 
the group orders are not divisible by any prime that splits in the real
field up to the CRT bound used.

In fact, any prime that divides the group order of a curve having RM 
by the maximal order of $\Q(\sqrt{\D})$ must either be a split (or 
ramified) prime, or divide it with multiplicity 2. 
As a consequence, the early abort strategy works much better than in 
the classical Schoof algorithm, because, it suffices to test half the 
number of primes up to our CRT bound. 

We ran a search for a secure curve over a prime field of 128 bits, using 
a CRT bound of 131.
Our series of computations frequently aborted early, and resulted in 245 
curves for which $\chi(T)$ was fully determined, and for which neither 
the group order nor its twist was divisible by a prime less than 131. 
Considering these twists, this provided 490 group orders, of which 27 were 
prime, and therefore 
suitable for cryptographic use.
We give here the data for one of these curves, that was furthermore
twist-secure: both the Jacobian and the twist Jacobian order are prime.

Let $\C/\F_q$, where $q=2^{128}+573$, be the curve in the family 
$\TTV{\family{C}}$ of \S\ref{sec:CTfamily} specialized to 
$t=75146620714142230387068843744286456025$.  The characteristic 
polynomial $\chi(T)$ is determined by 
$$ 
(s_1,s_2) = (-26279773936397091867,\;
-90827064182152428161138708787412643439),
$$
giving prime group orders for the Jacobian: % and for its twist:
$$
115792089237316195432513528685912298808995809621534164533135283195301868637471,
$$
and for its twist:
$$
115792089237316195414628441331463517678650820031857370801365706066289379517451.
$$
We note that correctness of the orders is easily verified on 
random points in the Jacobians. 

%%%%%%%%%%%%%%%%%%%%%%%%%%%%%%%%%%%%%%%%%%%%%%%%%%%%%%%%%%%%%%%%%%%%%%%%%%%%%%%%
%%%%%%%%%%%%%%%%%%%%%%%%%%%%%%%%%%%%%%%%%%%%%%%%%%%%%%%%%%%%%%%%%%%%%%%%%%%%%%%%
\subsection{A Kilobit Jacobian}
%%%%%%%%%%%%%%%%%%%%%%%%%%%%%%%%%%%%%%%%%%%%%%%%%%%%%%%%%%%%%%%%%%%%%%%%%%%%%%%%
%%%%%%%%%%%%%%%%%%%%%%%%%%%%%%%%%%%%%%%%%%%%%%%%%%%%%%%%%%%%%%%%%%%%%%%%%%%%%%%%

Let $q$ be the prime $2^{512}+1273$, 
and consider the curve over $\F_{q}$ from the family
$\TTV{\family{C}}$ of \S\ref{sec:CTfamily}
specialized at
\[
    \begin{array}{rl}
    t = & 
    2908566633378727243799826112991980174977453300368095776223\\ &
    2569868073752702720144714779198828456042697008202708167215\\ &
    32434975921085316560590832659122351278 .
    \end{array}
\]
%$t=2908566633378727243799826112991980174977453300368095776223256986807375270272014471477919882845604269700820270816721532434975921085316560590832659122351278$.
This value of $t$ was randomly chosen, and carries no special structure.
We computed the values of the pair $(s_1\bmod\ell,n\bmod\ell)$ for this curve
for each split prime $\ell$ up to $419$; this was enough to 
uniquely determine the true value of $(s_1, n)$ using the Chinese Remainder
Theorem. 
The numerical data for the curve follows:
\[
    \begin{array}{r@{\;}c@{\;}l}
    \D & = & 5 \\
    s_1 & = &  
    -10535684568225216385772683270554282199378670073368228748\\ & & 
    7810402851346035223080\\
    n & = &
    -37786020778198256317368570028183842800473749792142072230\\ & & 
    993549001035093288492\\
    s_2 & = & (s_1^2 - n^2\D)/4 \\
        & = &
    990287025215436155679872249605061232893936642355960654938\\ & & 
    008045777052233348340624693986425546428828954551752076384\\ & & 
    428888704295617466043679591527916629020\\
    \end{array}
\]
The order of the Jacobian is therefore
\[
    \begin{array}{r@{\;}c@{\;}l}
    N & = & (1+q)^2 - s_1(1+q) + s_2\\
        & = &
    179769313486231590772930519078902473361797697894230657273\\ & & 
    430081157732675805502375737059489561441845417204171807809\\ & & 
    294449627634528012273648053238189262589020748518180898888\\ & & 
    687577372373289203253158846463934629657544938945248034686\\ & & 
    681123456817063106485440844869387396665859422186636442258\\ & & 
    712684177900105119005520.
    \end{array}
\]
The total runtime for this computation was about 80 days on a single core
of a Core 2 clocked at 2.83 GHz. In practice, we use the inherent
parallelism of the algorithm, running one prime~$\ell$ on each available core.

We did not 
compute the characteristic polynomial modulo small prime powers
(as in~\cite{GaSc10}),
nor did we use BSGS to
deduce the result from partial modular information
as in \S\ref{sec:search-space}
(indeed, we were more
interested in measuring the behaviour of our algorithm for large values
of~$\ell$).
These improvements with an exponential-complexity 
nature bring much less than in the classical point
counting algorithms, since they have to be balanced with a
polynomial-time algorithm with a lower degree. 
For this example, we
estimate that BSGS and small prime powers could have saved a factor
of about 2 in the total runtime.

%%%%%%%%%%%%%%%%%%%%%%%%%%%%%%%%%%%%%%%%%%%%%%%%%%%%%%%%%%%%%%%%%%%%%%%%%%%%%%%%
%%%%%%%%%%%%%%%%%%%%%%%%%%%%%%%%%%%%%%%%%%%%%%%%%%%%%%%%%%%%%%%%%%%%%%%%%%%%%%%%
\subsection{Degrees of Division Polynomials}
%%%%%%%%%%%%%%%%%%%%%%%%%%%%%%%%%%%%%%%%%%%%%%%%%%%%%%%%%%%%%%%%%%%%%%%%%%%%%%%%
%%%%%%%%%%%%%%%%%%%%%%%%%%%%%%%%%%%%%%%%%%%%%%%%%%%%%%%%%%%%%%%%%%%%%%%%%%%%%%%%

For each prime $\ell$ splitting in $\Z[\TTV{\phi}]$, 
we report the degree of the $\alpha$-division polynomial $d_2$
(where $\alpha$ is the endomorphism of norm $\ell$ that was used).
By Lemma~\ref{lemma:generator-reduction}, 
$\deg(d_2)$ is in $O(\ell)$; 
the table below gives the ratio $\deg(d_2)/\ell$, 
thus measuring the hidden constant in the $O()$ notation. 
\begin{center}
\begin{tabular}{|c||c|c|c|c|c|c|c|c|c|c|c|c|c|c|c||}
\hline $\ell$ &
11 & 19 & 29 & 31 & 41 & 59 & 61 & 71 & 79 & 89 & 101 & 109 & 131 \\
\hline $\deg d_2 / \ell$ &
1.82 & 2.05 & 2.07 & 1.94 & 2.05 & 2.10 & 1.97 & 2.03 & 2.01 & 2.02 &
1.98 & 2.02 & 2.02 \\
\hline
\hline $\ell$ &
139 & 149 & 151 & 179 & 181 & 191 & 199 & 211 & 229 & 239 & 241 & 251 &
269 \\ 
\hline $\deg d_2 / \ell$ &
2.12 & 2.04 & 1.99 & 2.00 & 2.01 & 2.09 & 2.21 & 1.99 & 2.18 & 2.01 &
2.05 & 2.07 & 2.17 \\
\hline
\hline $\ell$ &
271 & 281 & 311 & 331 & 349 & 359 & 379 & 389 & 401 & 409 & 419 & & \\
\hline $\deg d_2 / \ell$ &
2.01 & 1.99 & 2.11 & 2.12 & 2.13 & 2.02 & 2.00 & 2.16 & 2.03 & 2.10 &
2.00 & & \\
\hline
\end{tabular}
\end{center}
We have 
$\deg(d_1) = \deg(d_2)+1$ and $\deg(d_0) = \deg(d_2)+2$. 
All of these degrees depend only on the curve family $\TTV{\family{C}}$,
and not on the individual curve chosen.

%%%%%%%%%%%%%%%%%%%%%%%%%%%%%%%%%%%%%%%%%%%%%%%%%%%%%%%%%%%

\end{document}